\documentclass[10pt]{article}
\usepackage{amsmath,amsfonts,amssymb,amsthm,graphicx}
\usepackage{etoolbox,mathtools}
\usepackage{xcolor}
\usepackage[colorlinks=true,citecolor=blue,pdfpagemode=UseNone,pdfstartview=FitH,pdfpagelabels]{hyperref}

\allowdisplaybreaks[4]

\newcommand{\Z}{\mathbb{Z}}

\renewcommand{\d}{\,\mathrm{d}}

\newcommand{\E}{\mathbb{E}}
\renewcommand{\P}{\mathbb{P}}

\newcommand{\Loss}{\mathrm{Loss}}

\newcommand{\DDD}{\mathcal{D}}
\newcommand{\FFF}{\mathcal{F}}

\theoremstyle{plain}
\newtheorem{theorem}{Theorem}
\newtheorem{corollary}[theorem]{Corollary}

\newtheorem{proposition}[theorem]{Proposition}

\theoremstyle{definition}

\theoremstyle{remark}
\newtheorem{remark}[theorem]{Remark}

\newcommand{\abstr}{This note proves a law of large numbers for predicting several steps ahead,
  which, in the case of uniformly bounded random variables,
  generalizes the standard law of large numbers for martingales;
  the standard law of large numbers corresponds to predicting one step ahead.
  Its main result shows that the law of large numbers holds
  for predicting $N$ uniformly bounded random variables
  $o(N)$ steps ahead, but it is much more precise and in some respects optimal.
  This law of large numbers is  applied to a problem of decision making
  with a bounded loss function limiting the impact of each decision
  to $o(N)$ steps.}

\title{A law of large numbers for predicting several steps ahead}
\author{Vladimir Vovk}

\begin{document}
  \maketitle
  \begin{abstract}
    \smallskip
    \abstr

    The version of this note at \url{http://gtfp.net} (Working Paper 67)
    is updated most often.
  \end{abstract}

\section{Introduction}

The usual martingale statements of limit theorems of probability theory
involve one-step-ahead averages $\E(Y_n\mid\FFF_{n-1})$,
where $Y_n$ is an adapted sequence for a filtration $(\FFF_n)$.
A typical law of large numbers says that,
under some conditions (such as $Y_n$ being uniformly bounded),
\begin{equation}\label{eq:unbiased}
  \frac1N
  \sum_{n=1}^N
  \left(
    Y_n - \E(Y_n\mid\FFF_{n-1})
  \right)
  \approx
  0
\end{equation}
with high probability,
provided $N$ is large enough
(see, e.g., \cite[Corollary~2 in Section~7.3]{Shiryaev:2019}).
This property may be expressed by saying that the one-step-ahead forecasts
$\E(Y_n\mid\FFF_{n-1})$ for  $Y_n$
are asymptotically unbiased.
It is easy to see that \eqref{eq:unbiased} continues to hold
for predicting $K$ steps ahead when $K$ is a small constant (while $N$ is large):
\begin{equation}\label{eq:unbiased-K}
  \frac1N
  \sum_{n=1}^N
  \left(
    Y_n - \E(Y_n\mid\FFF_{n-K})
  \right)
  \approx
  0.
\end{equation}
The question asked in this note is how large $K$ can be
in order to have \eqref{eq:unbiased-K} with high probability
for uniformly bounded $Y_n$.
A crude answer is that, for uniformly bounded $Y_n$,
it can be of the order $o(N)$ but no more in general.

Section~\ref{sec:theorem} states the main positive result of this note,
while Section~\ref{sec:asymptotic} explores its optimality.
To demonstrate the usefulness of our law of large numbers,
Section~\ref{sec:decision} applies it to a simple problem of decision making.
Section~\ref{sec:conclusion} concludes and lists some directions of further research.
Appendix~\ref{app:optimality} complements Section~\ref{sec:asymptotic}.

\section{A law of large numbers in the form of inequality}
\label{sec:theorem}

We consider a two-sided \emph{filtration} $(\FFF_n)$, $n\in\Z:=\{\dots,-1,0,1,\dots\}$,
on a probability space $(\Omega,\FFF,\P)$;
it is an increasing sequence of sub-$\sigma$-algebras of $\FFF$,
i.e., $\FFF_{n}\subseteq\FFF_{n'}$ whenever $n\le n'$.
As usual, a sequence $(Y_n)$ of random variables in $(\Omega,\FFF,\P)$ is \emph{adapted}
if $Y_n$ is $\FFF_n$-measurable for all $n$.
For simplicity we consider uniformly bounded $Y_n$
and take, without loss of generality, 1 to be an upper bound for $\lvert Y_n\rvert$.
The expectation symbol $\E$ always refers
to the underlying probability space $(\Omega,\FFF,\P)$.

\begin{theorem}\label{thm:LLN}
  Let $(\Omega,\FFF,\P)$ be a probability space
  equipped with a filtration $(\FFF_n)$, $n\in\Z$.
  Fix $K\in\{1,2,\dots\}$.
  Let $Y_1,\dots,Y_N$ be an adapted sequence of random variables
  in $(\Omega,\FFF,\P)$
  bounded by 1 in absolute value, $\lvert Y_n\rvert\le1$ for $n=1,\dots,N$.
  Then we have, for any $\epsilon>0$,
  \begin{equation}\label{eq:LLN}
    \P
    \left(
      \left|
        \sum_{n=1}^N
        \bigl(
          Y_n - \E(Y_n\mid\FFF_{n-K})
        \bigr)
      \right|
      \ge
      \sqrt{8K(N+K)\ln\frac{2}{\epsilon}}
    \right)
    <
    \epsilon.
  \end{equation}
\end{theorem}

We will sometimes refer to $K$ as our \emph{prediction horizon}.
Our interpretation of \eqref{eq:LLN} is that
$K$-steps-ahead forecasts $\E(Y_n\mid\FFF_{n-K})$ for $Y_n$
are asymptotically unbiased when $K\ll N$.
Our proof will use the primitive idea of decomposing forecasting $K$ steps ahead
into $K$ processes of forecasting one step ahead,
each of the $K$ processes paying attention only to every $K$th observation $Y_n$.
Interestingly, this will give nearly optimal results,
as we will see in Section~\ref{sec:asymptotic} (and then Appendix~\ref{app:optimality}).
To merge performance guarantees for the $K$ processes
we will need one result from robust risk aggregation,
namely, \cite[Theorem~4.2]{Embrechts/Puccetti:2006}.

\begin{proof}[Proof of Theorem~\ref{thm:LLN}]
  It suffices to prove that
  \begin{equation}\label{eq:old-LLN}
    \P
    \left(
      \left|
        \sum_{n=1}^N
        \bigl(
          Y_n - \E(Y_n\mid\FFF_{n-K})
        \bigr)
      \right|
      \ge
      \sqrt{8K N\ln\frac{2}{\epsilon}}
    \right)
    <
    \epsilon
  \end{equation}
  under the assumption that $N$ is divisible by $K$.
  Indeed, if $N$ is not, we can replace it by $N':=\lceil N/K\rceil K$
  and apply \eqref{eq:old-LLN} to $N'$ in place of $N$
  and $Y_n$, $n=N+1,\dots,N'$, set to 0.

  We will need the following special case of Theorem~4.2
  in \cite{Embrechts/Puccetti:2006}.
  Suppose nonnegative random variables $X_k$, $k=1,\dots,K$,
  satisfy
  \begin{equation}\label{eq:marginals}
    \P(X_k\ge x)
    =
    \exp(-a x^2)
  \end{equation}
  for all $x\ge0$,
  where $a$ is a positive constant.
  Let $C>0$.
  The value $E$ of the optimization problem
  \begin{equation*}
    \P(X_1+\dots+X_K\ge C)
    \to
    \max
  \end{equation*}
  (the $\max$, or at least $\sup$,
  being over all joint distributions for $(X_1,\dots,X_K)$
  with the given marginals)
  does not exceed
  \begin{equation}\label{eq:E}
    \inf_{t\in(0,C/K)}
    \frac{K\int_t^{C-(K-1)t}\exp(-a x^2)\d x}{C-K t}
    \le
    \inf_{t\in(0,C/K)}
    \frac{K\int_t^{\infty}\exp(-a x^2)\d x}{C-K t}.
  \end{equation}

  We can extend the statement in the previous paragraph
  to a wider class of random variables $X_k$, $k=1,\dots,K$.
  Namely, it suffices to assume that they satisfy
  \begin{equation}\label{eq:relaxed-marginals}
    \P(X_k\ge x)
    \le
    \exp(-a x^2)
  \end{equation}
  for all $x\ge0$, instead of \eqref{eq:marginals}
  (if needed, we can increase such $X_k$,
  perhaps extending the underlying probability space,
  to make sure \eqref{eq:marginals} holds).
  We will apply the statement to the random variables $X_k$ given by
  \begin{equation*}
    X_k
    :=
    \sum_{n\in\{k,k+K,k+2K,\dots,k+(N/K-1)K\}}
    \left(
      Y_{n} - \E(Y_{n}\mid\FFF_{n-K})
    \right).
  \end{equation*}
  By Hoeffding's inequality,
  for any $c>0$ and any $k\in\{1,\dots,K\}$,
  \begin{equation*}
    \P
    \left(
      X_k
      \ge
      c
    \right)
    \le
    \exp
    \left(
      -\frac{c^2}{2N/K}
    \right)
  \end{equation*}
  (see, e.g., the statement of Hoeffding's inequality
  in \cite[Section~A.6.3]{Vovk/etal:2022book},
  which only requires the martingale difference to lie in a predictable interval
  of a given length).
  Therefore, \eqref{eq:relaxed-marginals} holds with
  \begin{equation}\label{eq:a}
    a
    :=
    \frac{K}{2N}.
  \end{equation}

  Let us set $t:=\frac{C}{2K}$ in \eqref{eq:E}
  (this is the middle of the range of $t$).
  This gives the upper bound
  \[
    \frac{2K}{C}
    \int_{\frac{C}{2K}}^{\infty}
    \exp(-a x^2)
    \d x
  \]
  for $E$,
  which can be rewritten (see below for an explanation) as
  \begin{align}
    &\frac{2K}{C}
    \frac{1}{\sqrt{2a}}
    \int_{\sqrt{2a}\frac{C}{2K}}^{\infty}
    \exp(-y^2/2)
    \d y
    =
    \frac{2K}{C}
    \frac{\sqrt{2\pi}}{\sqrt{2a}}
    \bar\Phi
    \left(
      \sqrt{2a}\frac{C}{2K}
    \right)\label{eq:notation}\\
    &=
    \frac{2\sqrt{2\pi}\sqrt{K N}}{C}
    \bar\Phi
    \left(
      \frac{C}{2\sqrt{K N}}
    \right)
    <
    \frac{4 K N}{C^2}
    \exp
    \left(
      -\frac{C^2}{8 K N}
    \right)\label{eq:Feller}.
  \end{align}
  The first expression in~\eqref{eq:notation}
  is obtained by the substitution $y:=\sqrt{2a}x$,
  the equality in \eqref{eq:notation} uses the notation $\bar\Phi$
  for the survival function of the standard Gaussian distribution,
  the following equality (the one in~\eqref{eq:Feller})
  is obtained by plugging in \eqref{eq:a},
  and the final inequality in~\eqref{eq:Feller}
  follows from the usual upper bound for $\bar\Phi$
  \cite[Lemma~VII.1.2]{Feller:1968}.

  To find a suitable solution to the inequality
  \begin{equation}\label{eq:suitable}
    \frac{4 K N}{C^2}
    \exp
    \left(
      -\frac{C^2}{8 K N}
    \right)
    <
    \frac{\epsilon}{2},
  \end{equation}
  we plug in $C=\sqrt{8 K N x \ln\frac{2}{\epsilon}}$
  (intuitively, $x$ should not be so different from $1$)
  obtaining, after simplification,
  \[
    \epsilon^{x-1}
    <
    2^x x \ln\frac{2}{\epsilon}.
  \]
  We can set $x:=1$,
  which gives \eqref{eq:old-LLN}
  (this assumes $\epsilon\in(0,1]$,
  but in the case $\epsilon>1$ \eqref{eq:old-LLN} holds trivially).
\end{proof}

\begin{remark}
  A typical statement of the law of large numbers
  replaces the sum in \eqref{eq:LLN} by $\sum_{n=1}^N Y_n$
  where $Y_n$ satisfy $\E(Y_n\mid\FFF_{n-K})=0$.
  However, such a replacement would somewhat weaken Theorem~\ref{thm:LLN}.
  Indeed, while \eqref{eq:LLN} immediately implies
  \begin{equation}\label{eq:LLN-reduced}
    (\forall n: \E(Y_n\mid\FFF_{n-K})=0)
    \Longrightarrow
    \P
    \left(
      \left|
        \sum_{n=1}^N Y_n
      \right|
      \ge
      \sqrt{8K(N+K)\ln\frac{2}{\epsilon}}
    \right)
    <
    \epsilon,
  \end{equation}
  \eqref{eq:LLN-reduced} does not imply \eqref{eq:LLN},
  since centring a random variable (by subtracting its conditional expectation)
  can change its range $[-1,1]$ (and thus require replacing the 8 by a larger constant).
\end{remark}

\begin{remark}\label{rem:LLN}
  The proof of Theorem~\ref{thm:LLN} also demonstrates
  the following one-sided counterpart of \eqref{eq:LLN-reduced}:
  \begin{equation*}
    \P
    \left(
      \sum_{n=1}^N Y_n
      \ge
      \sqrt{8K(N+K)\ln\frac{2}{\epsilon}}
    \right)
    <
    \frac{\epsilon}{2}
  \end{equation*}
  provided $\E(Y_n\mid\FFF_{n-K})\le0$ for all $n$ a.s.
  Indeed, the inequality \eqref{eq:suitable}
  shows that \eqref{eq:LLN} continues to hold when the vertical bars
  are dropped in $\lvert\dots\rvert$
  and the $\epsilon$ on the right-hand side is replaced by $\frac{\epsilon}{2}$;
  we can then remove $\E(Y_n\mid\FFF_{n-K})$
  as $\E(Y_n\mid\FFF_{n-K})\le0$ a.s.
\end{remark}

\begin{remark}
  In the proof of Theorem~\ref{thm:LLN}
  we did not make any attempt to optimize
  the factor $8$ in \eqref{eq:LLN}.
  However, a similar argument shows that 8
  can be replaced by a number as close to $2$ as we wish
  if we narrow down the permitted range of $\epsilon$
  (leaving the lower end of the range at 0, of course).
  For example, for $\epsilon\in(0,0.1]$, we have
  \begin{equation*}
    \P
    \left(
      \left|
        \sum_{n=1}^N
        \bigl(
          Y_n - \E(Y_n\mid\FFF_{n-K})
        \bigr)
      \right|
      \ge
      \sqrt{2.73 K(N+K)\ln\frac{2}{\epsilon}}
    \right)
    <
    \epsilon
  \end{equation*}
  in place of \eqref{eq:LLN} (proof omitted).
\end{remark}

\section{An asymptotic statement and its optimality}
\label{sec:asymptotic}

Let us state Theorem~\ref{thm:LLN} in a cruder way
(traditional for stating the law of large numbers).
According to~\eqref{eq:LLN},
\begin{equation}\label{eq:LLN-simple}
  \left|
    \sum_{n=1}^N
    \left(
      Y_{n} - \E(Y_{n}\mid\FFF_{n-K})
    \right)
  \right|
  =
  O_p
  \left(
    \sqrt{K N}
  \right)
\end{equation}
(which assumes $K\le N$).
The usual statement of the law of large numbers
\begin{equation*}
  \left|
    \frac1N
    \sum_{n=1}^N
    \left(
      Y_{n} - \E(Y_{n}\mid\FFF_{n-K})
    \right)
  \right|
  =
  o_p(1)
\end{equation*}
follows when $K=o(N)$.

The following proposition is an inverse to \eqref{eq:LLN-simple}.

\begin{proposition}\label{prop:anti-LLN-1}
  Suppose the underlying probability space $(\Omega,\FFF,\P)$ is atomless.
  For any $N$ and any prediction horizon $K\le N$
  there exists a sequence $Y_1,\dots,Y_N$ of random variables
  that are bounded by 1 in absolute value, $\lvert Y_n\rvert\le1$ for $n=1,\dots,N$,
  such that
  \begin{equation}\label{eq:anti-LLN-1}
    \P
    \left(
      \sum_{n=1}^N
      \bigl(
        Y_{n} - \E(Y_{n}\mid\FFF_{n-K})
      \bigr)
      \ge
      \sqrt{K(N-K)}
    \right)
    \ge
    0.1,
  \end{equation}
  where $(\FFF_n)$ is the filtration generated by $(Y_n)$
  (i.e., $\FFF_n:=\sigma(Y_1,\dots,Y_n)$,
  meaning $\FFF_n:=\{\emptyset,\Omega\}$ for $n<1$).
\end{proposition}

\noindent
The proof will shows that,
when we replace the sum in \eqref{eq:anti-LLN-1}
by its absolute value, as in \eqref{eq:LLN},
we can replace the $0.1$ on the right-hand side of \eqref{eq:anti-LLN-1}
by $0.2$.

\begin{proof}[Proof of Proposition~\ref{prop:anti-LLN-1}]
  Let us assume that $N$ is divisible by $K$ and prove
  \begin{equation}\label{eq:anti-LLN-1-simple}
    \P
    \left(
      \sum_{n=1}^N
      \left(
        Y_{n} - \E(Y_{n}\mid\FFF_{n-K})
      \right)
      \ge
      \sqrt{K N}
    \right)
    \ge
    0.1
  \end{equation}
  for some choice of $Y_n$
  (this will then imply~\eqref{eq:anti-LLN-1} without this restriction
  by applying \eqref{eq:anti-LLN-1-simple} to $N':=\lfloor N/K\rfloor K$
  in place of $N$).
  Set $m:=N/K$.
  Fix independent $\{-1,1\}$-valued variables $X_1,\dots,X_m$
  taking values $\pm1$ with equal probabilities
  (they exist by, e.g.,
  \cite[Lemma~D.1 in the Online Supplement]{Vovk/Wang:2021}),
  and set
  \begin{equation*}
    Y_{n}
    :=
    X_{\lceil n/K\rceil},
    \quad
    n=1,\dots,N.
  \end{equation*}
  Therefore, the $N$ steps are split into $m$ blocks of length $K$,
  and $Y_{n}$ is constant within each block.
  By the central limit theorem applied to $X_1,\dots,X_m$,
  we have $Y_{n}=1$ in at least $\sqrt{m}$ more blocks than $Y_{n}=-1$
  with probability $\Phi(-1)\approx0.159$ in the limit as $m\to\infty$,
  where $\Phi$ is the standard Gaussian distribution function.
  The smallest value of this probability for finite $m$ is smaller,
  approximately $0.109$, and it is attained at $m=6$.
  If there is such an imbalance of at least $\sqrt{m}$,
  \[
    \sum_{n=1}^N
    \left(
      Y_{n} - \E(Y_{n}\mid\FFF_{n-K})
    \right)
    =
    \sum_{n=1}^N Y_{n}
    \ge
    K\sqrt{m}
    =
    K \sqrt{N/K}
    =
    \sqrt{K N},
  \]
  which agrees with \eqref{eq:anti-LLN-1-simple}.
\end{proof}

\section{An application to decision making}
\label{sec:decision}

In this section we will apply Theorem~\ref{thm:LLN}
to a problem of decision making similar to the one
considered in \cite[Section~6]{Vovk:2026-local}.
At each step $n$, $n=1,2,\dots$,
a decision maker is required to make a decision $d\in\mathbf{D}$,
chosen from a measurable space $(\mathbf{D},\DDD)$.
The loss suffered by the decision maker
as a result of making the decision $d$ at step $n$
is measured by a measurable \emph{loss function}
$\lambda_n:\Omega\times\mathbf{D}\to[0,1]$.
The loss functions $\lambda_n$ are assumed to be uniformly bounded,
and without loss of generality, we take their range to be $[0,1]$.
Our notation for the random loss resulting from decision $d$ taken at step $n$
is $\lambda_n(d)$, so that $\lambda_n(d)(\omega):=\lambda_n(\omega,d)$.

We are interested in suitable strategies for the decision maker,
referring to them as decision strategies.
Formally, a \emph{decision strategy} $A$
is a sequence of adapted $\mathbf{D}$-valued random elements $A_n$;
namely, each $A_n:\Omega\to\mathbf{D}$ is assumed to be $\FFF_n/\DDD$-measurable.
Let us assume that there exists a \emph{Bayesian strategy} $B$,
i.e., a decision strategy satisfying
\begin{equation*}
  \E(\lambda_n(B_n)\mid\FFF_n)
  \le
  \E(\lambda_n(A_n)\mid\FFF_n)
  \quad
  \text{a.s.}
\end{equation*}
for any $n$ and any other decision strategy $A$.
When $\mathbf{D}$ is finite,
the existence of a Bayesian strategy is automatic:
we can define
\[
  B_n
  \in
  \arg\min_{d\in\mathbf{D}}
  \E(\lambda_n(d)\mid\FFF_n),
\]
selecting the first decision $d$ in a fixed ordering of $\mathbf{D}$
if the $\arg\min$ contains more than one element.
But in general, the existence of a Bayesian strategy has to be assumed.
In interesting cases it usually exists and chooses an optimal decision at each step.
The total loss of a decision strategy $A$ over the first $N$ steps is denoted by
\[
  \Loss_N(A)
  :=
  \sum_{n=1}^N
  \lambda_n(A_n).
\]

Let us fix a Bayesian strategy $B$.
In order to show that the Bayesian strategy is better, or at least not much worse,
than any alternative decision strategy,
we have to assume that the loss resulting from the decision $d$ made at step $n$
does not depend substantially on remote future
(if it does, our task is in general hopeless).
Namely, we assume that $\lambda_n(d)$ becomes determined at step $n+K$
for some parameter $K\in\{1,2,\dots\}$,
which we will call the \emph{impact horizon},
so that the effect of $d$ disappears after step $n+K$;
the role of the impact horizon is analogous to that of a prediction horizon.
Formally, $\lambda_n:\Omega\times\mathbf{D}\to[0,1]$
is assumed to be $\FFF_{n+K}\times\DDD$-measurable.

\begin{corollary}\label{cor:LLN}
  The Bayesian strategy $B$ satisfies,
  for any $\epsilon>0$, any decision strategy $A$, and any $N$,
  \begin{equation}\label{eq:cor}
    \P
    \left(
      \Loss_N(B) - \Loss_N(A)
      \ge
      \sqrt{8K(N+K)\ln\frac{2}{\epsilon}}
    \right)
    <
    \epsilon.
  \end{equation}
\end{corollary}

According to \eqref{eq:cor}, the Bayesian strategy's total loss
over the first $N$ steps
(the total loss being determined only after $n+K$ steps)
is, with high probability, almost as low as any other decision strategy's
for a short impact horizon, $K\ll N$.
This result is a measure-theoretic version of \cite[Theorem~6.5]{Vovk:2026-local}.

\begin{proof}[Proof of Corollary~\ref{cor:LLN}]
  Let us set
  $
    Y_{n+K}
    :=
    \lambda_n(B_n) - \lambda_n(A_n)
  $,
  where we use the notation $Y_{n+K}$ rather than $Y_n$
  to make this random sequence adapted.
  By the definition of a Bayesian strategy, we have
  $
    \E(Y_{n+K}\mid\FFF_{n}) \le 0
  $.
  Applying the one-sided law of large numbers as stated in Remark~\ref{rem:LLN}
  but with $Y_{1+K},\dots,Y_{n+K}$ in place of $Y_{1},\dots,Y_{n}$,
  we obtain \eqref{eq:cor} (even with $\epsilon/2$ in place of $\epsilon$
  on the right-hand side).
\end{proof}

\section{Conclusion}
\label{sec:conclusion}

This note proves a version of the law of large numbers
for predicting several steps ahead
and applies it to a problem of decision making
with a limited impact horizon.
These are some obvious directions of further research:
\begin{itemize}
\item
  get rid of the assumption that the random variables being averaged
  are uniformly bounded;
\item
  find the optimal numerical constants in Theorem~\ref{thm:LLN};
\item
  establish tighter lower bounds corresponding to the law of large numbers
  (a step in this direction is made in Appendix~\ref{app:optimality});
\item
  establish versions of other limit theorems of probability theory
  (such as the strong law of large numbers, law of the iterated logarithm,
  and central limit theorem)
  for predicting several steps ahead,
  perhaps with an increasing prediction horizon
  in the case of strong limit theorems;
\item
  replace the assumption of a limited impact horizon in decision problems
  by softer assumptions of decaying impact.
\end{itemize}

\subsection*{Acknowledgments}

Many thanks to Glenn Shafer for his help and encouragement
and to the anonymous referees of the journal version
for useful advice about presentation.

\appendix
\section{Optimality of Theorem~\ref{thm:LLN} in \texorpdfstring{$\epsilon$}{epsilon}}
\label{app:optimality}

Proposition~\ref{prop:anti-LLN-1}
states the optimality of the dependence of the right-hand sides
of \eqref{eq:LLN-simple} and of the inner inequality of \eqref{eq:LLN}
on $K$.
Another natural question is whether the dependence of the right-hand side
of the inner inequality of \eqref{eq:LLN} on $\epsilon$ is optimal.
A positive answer is provided by the following result,
which is, however, somewhat more difficult to state and interpret
than Proposition~\ref{prop:anti-LLN-1}.

\begin{proposition}\label{prop:anti-LLN-2}
  Suppose the underlying probability space is atomless.
  For any $N$, any prediction horizon $K$, and any $\epsilon>0$
  such that
  \begin{equation}\label{eq:integer}
    m := \frac{N}{K} \in 2\Z,
    \quad
    \sqrt{m\ln\frac{1}{15\epsilon}}
    \in
    4\Z,
  \end{equation}
  and
  \begin{equation}\label{eq:not-too-good}
    \frac12
    \sqrt{K N \ln\frac{1}{15\epsilon}}
    \le
    \frac{N}{4},
  \end{equation}
  there exists a sequence $Y_1,\dots,Y_N$ of random variables
  that are bounded by 1 in absolute value and satisfy
  \begin{equation}\label{eq:anti-LLN-2}
    \P
    \left(
      \sum_{n=1}^N
      \bigl(
        Y_{n} - \E(Y_{n}\mid\FFF_{n-K})
      \bigr)
      \ge
      \frac12
      \sqrt{K N \ln\frac{1}{15\epsilon}}
    \right)
    \ge
    \epsilon,
  \end{equation}
  where $(\FFF_n)$ in \eqref{eq:anti-LLN-2}
  is the filtration generated by $(Y_n)$.
\end{proposition}

If we ignore numerical constants
such as 8 in \eqref{eq:LLN} and $\frac12$ in \eqref{eq:anti-LLN-2},
the inequality~\eqref{eq:anti-LLN-2} can be considered to be an inverse to~\eqref{eq:LLN}.
When \eqref{eq:not-too-good} is violated,
the inner inequality in \eqref{eq:anti-LLN-2} expresses an extreme bias of the predictions,
namely, it implies
\[
  \frac1N
  \sum_{n=1}^N
  \bigl(
    Y_{n} - \E(Y_{n}\mid\FFF_{n-K})
  \bigr)
  >
  \frac{1}{4};
\]
the condition \eqref{eq:not-too-good} thus means that our result falls short
of dealing with such an extreme bias.
The conditions \eqref{eq:integer} still leave us
with a fairly dense net of permitted triples $(N,K,\epsilon)$.

\begin{proof}[Proof of Proposition~\ref{prop:anti-LLN-2}]
  We will modify the proof of Proposition~\ref{prop:anti-LLN-1}
  by applying a lower bound for large deviations
  in the form of \cite[Proposition 7.3.2]{Matousek/Vondrak:2008}.
  Define $X_1,\dots,X_m$ and $Y_1,\dots,Y_N$ as before.
  Let $Z$ be the number of times that $X_i=1$, $i=1,\dots,m$.
  Proposition 7.3.2 in \cite{Matousek/Vondrak:2008} then says that
  \begin{equation}\label{eq:MV-1}
    \P
    \left(
      Z \ge \frac{m}{2}+t
    \right)
    \ge
    \frac{1}{15}
    \exp
    \left(
      -16 t^2 / m
    \right)
  \end{equation}
  provided $m$ is even and $t$ is an integer in the range $[0,m/8]$.
  Since
  \[
    \sum_{n=1}^N
    \bigl(
      Y_{n} - \E(Y_{n}\mid\FFF_{n-K})
    \bigr)
    =
    (2Z-m)K,
  \]
  we can rewrite \eqref{eq:MV-1} as
  \begin{equation}\label{eq:MV-2}
    \P
    \left(
      \sum_{n=1}^N
      \bigl(
        Y_{n} - \E(Y_{n}\mid\FFF_{n-K})
      \bigr)
      \ge
      2 K t
    \right)
    \ge
    \frac{1}{15}
    \exp
    \left(
      -16 t^2 / m
    \right).
  \end{equation}
  If $\epsilon>0$ is equal to the right-hand side of \eqref{eq:MV-2},
  we have
  \begin{equation*}
    t
    =
    \frac14
    \sqrt{m\ln\frac{1}{15\epsilon}},
  \end{equation*}
  and plugging thus expression for $t$ into \eqref{eq:MV-2}
  gives \eqref{eq:anti-LLN-2}.

  The two conditions in \cite[Proposition 7.3.2]{Matousek/Vondrak:2008}
  are reflected in our statement:
  $t\le m/8$ becomes~\eqref{eq:not-too-good},
  and $m$ being even and $t$ being an integer are required in \eqref{eq:integer}.
\end{proof}

\begin{remark}
  Kunsch and Rudolf \cite[Lemma~3]{Kunsch/Rudolf:2019}
  slightly improve the constants in \cite[Proposition 7.3.2]{Matousek/Vondrak:2008},
  and using their result allows us to rewrite \eqref{eq:anti-LLN-2} in the form
  \begin{equation*}
    \P
    \left(
      \sum_{n=1}^N
      \bigl(
        Y_{n} - \E(Y_{n}\mid\FFF_{n-K})
      \bigr)
      \ge
      0.6
      \sqrt{K N \ln\frac{1}{4.3\epsilon}}
    \right)
    \ge
    \epsilon,
  \end{equation*}
  with the corresponding modifications of the conditions \eqref{eq:integer}
  and \eqref{eq:not-too-good}.
\end{remark}
\end{document}